\documentclass[12pt,english,reqno]{amsart}
\usepackage{geometry}
\usepackage{amsmath,amssymb,amsthm,bbm,mathtools,comment}
\usepackage{amsfonts}
\usepackage[shortlabels]{enumitem}
\usepackage[pdftex,colorlinks,backref=page,citecolor=blue]{hyperref}
\usepackage[mathscr]{euscript}
\usepackage[usenames,dvipsnames]{color}
\usepackage{tikz,babel,adjustbox}
\usepackage[numbers]{natbib}
\usepackage{graphicx}
\usepackage{caption}
\usepackage{subcaption}
\usepackage{verbatim}
\usepackage{array}
\usepackage[frame,cmtip,arrow,matrix,line,graph,curve]{xy}
\usepackage{etoolbox}
\usepackage{pifont}
\usepackage[final]{microtype}
\usepackage{cmtiup}

\hypersetup{pdfpagemode=UseNone,pdfstartview={XYZ null null 1.00}}
\usetikzlibrary{shapes.misc,calc,intersections,patterns,decorations.pathreplacing}
\usetikzlibrary{arrows,arrows.meta,shapes,positioning,decorations.markings}
\tikzstyle arrowstyle=[scale=1]

\allowdisplaybreaks

\makeatletter
\def\@settitle{\begin{center}%
		\bfseries\Large
		\@title
	\end{center}%
}
\patchcmd{\@setauthors}{\MakeUppercase}{\normalsize}{}{}
\makeatother

\theoremstyle{plain}
\newtheorem{theorem}{Theorem}[section]

\newtheorem{claim}[theorem]{Claim}

\theoremstyle{remark}

\newcommand{\beq}[1]{\begin{equation}\label{#1}}
\newcommand{\enq}[0]{\end{equation}}

\makeatletter
\def\imod#1{\allowbreak\mkern10mu({\operator@font mod}\,\,#1)}
\makeatother

\newcommand{\Awin}{\mathsf{Alice}}
\newcommand{\Bwin}{\mathsf{Bob}}

\newcommand{\cF}{\mathcal F}
\newcommand{\cG}{\mathcal G}
\newcommand{\cT}{\mathcal T}

\begin{document}

\title{How to pick your football team}
\author{Bhargav Narayanan}
\address{Department of Mathematics, Rutgers University, Piscataway, NJ 08854, USA}
\email{narayanan@math.rutgers.edu}

\date{24 April 2026}
\dedicatory{To Imre Leader, for teaching me the pleasure of thinking very hard about very silly things.}

\begin{abstract}
	Team captains Alice and Bob divide up $2m$ footballers, each reduced to a real-valued score, into two teams of $m$ footballers each. On each turn, one captain plays \emph{picker}, and the other \emph{chooser}: the picker names a footballer yet to be selected, and the chooser decides which captain's team receives that footballer. Alice starts as picker, Bob as chooser, and roles alternate. The game ends as soon as either captain has a full team of $m$ footballers, at which point the other team receives all remaining footballers. The team with the larger sum of scores wins. Settling a problem raised by Eccles in 2015, we show that Alice cannot win.
\end{abstract}

\maketitle

\section{Introduction}
We study a class of combinatorial games inspired by the two-player game of \emph{football} played on a board of $2m$ real numbers. In this game (and all the others we study), Alice is the picker on the first turn, Bob is the chooser, and the players switch roles after each turn. On each turn, the picker names an element still on the board; the chooser removes that element from the board and gives it to a player of their choice. The game stops as soon as one player has received $m$ elements, at which point the other player receives all remaining elements. Thus, both players end the game with $m$ elements, and the player with the larger sum wins. Specialising our main result to football confirms an intriguing possibility suggested by Eccles~\citep{Eccles2015}.

\begin{theorem}\label{thm:main-football}
	Alice cannot win the game of football on any board of even cardinality.
\end{theorem}

This result has a curious history. Eccles~\citep{Eccles2015} communicated it to the author as a puzzle, but after several mathematicians failed to find a solution, the author learnt (at the prompting of a few colleagues) that Eccles had intended Theorem~\ref{thm:main-football} as a conjecture, anticipating the case of the special board $\{1,2,\ldots,2m\}$ to generalise. That special case is easier, since Bob can exploit the natural pairing ($i$ and $2m+1-i$) to avoid defeat, but Theorem~\ref{thm:main-football}, in full generality, does not prove itself. Indeed, this result captures too narrow a slice of a more general phenomenon, so we will need to prove something stronger. To motivate our main result, we briefly discuss two of the difficulties with analysing football.

The first issue is that there are many nonisomorphic \emph{winning families}. For example, Bob's optimal strategy on a four-element board $\{x_1\le x_2\le x_3\le x_4\}$ is determined by how $x_1+x_4$ and $x_2+x_3$ compare. If $x_1+x_4 \ge x_2+x_3$, then Bob aims to (and can) keep $x_4$ out of Alice's hand. If $x_1+x_4 \le x_2+x_3$, then Bob plays to ensure $x_1$ ends up in Alice's hand. In either case, it is not hard to see that Bob can avoid defeat. More generally, as the ordered board $\{x_1 \le \dots \le x_{2m}\}$ varies, the winning family $\cF$, consisting of those subsets $S \subset [2m]$ of size $m$ for which
\[
	\sum_{i \in S} x_i > \sum_{i \in S^c} x_i,
\]
has $2^{\Theta(m)}$ nonisomorphic possibilities, so $\cF$ can be fairly complicated in general. Indeed, analysing the six-element board should convince the motivated reader that there is no simple description of Bob's optimal strategy on every board.

This `explosion' of the game tree is familiar in the broader theory of games. Beck~\citep{Beck2008} gives a broad account of the combinatorial circle of ideas --- pairing, potential functions, strategy stealing, and the probabilistic method --- used to address this issue; for a more algebraic approach to the theory (and other mathematical games of football), see Berlekamp--Conway--Guy~\citep{BerlekampConwayGuy2001}. What appears to be new here --- as well as the primary source of difficulty --- is the alternation of roles between the players.

The second issue concerns the breaking of symmetry. For example, after five turns in football, it can happen that Alice has received four of the five elements that have gone off the board, and Bob only one. We are therefore forced to consider a larger class of games where Alice and Bob divide the board asymmetrically between themselves using the same mechanics. However, unlike in the symmetric situation, it is not clear how to compare the relative strengths of Alice and Bob in this more general setting; no numerical payoff (like the sum or average) appears to be suitable.

To address these issues, we forget the additive structure of the winning families, retaining only the fact that such families are monotone, and we estimate the relative strengths of Alice and Bob by comparing their positions in different games. To make these ideas precise and state our main result, we need to establish some notation, a task to which we now turn.

With any family $\cF \subset \binom{[n]}{k}$ of $k$-sets, we may associate two natural \emph{picker--chooser} games on the board $[n]$. In both games, there is a protagonist and an antagonist. At the end of the game, the protagonist holds a $k$-set and the antagonist holds an $(n-k)$-set.

Throughout, \emph{Alice picks first} and \emph{Bob chooses first}. The game mechanics are as before: on each turn, one player is the picker and the other is the chooser; the picker names an element still on the board, the chooser takes the element off the board and decides who receives it, and the roles of picker and chooser alternate. The game stops as soon as either the protagonist has $k$ elements or the antagonist has $n-k$ elements, at which point the other player receives all remaining elements.

In \emph{Alice's $\cF$-game}, Alice is the protagonist, and Alice wins if her final $k$-set lies in $\cF$; similarly, in \emph{Bob's $\cF$-game}, Bob is the protagonist, and Bob wins if his final $k$-set lies in $\cF$.

Next, we say that a family $\cF \subset \binom{[n]}{k}$ is \emph{pointwise-increasing}, or simply \emph{increasing}, if it is closed upwards in the pointwise order, i.e., in the order in which $\{s_1<\cdots<s_k\} \preceq \{t_1<\cdots<t_k\}$ if and only if $s_i\le t_i$ for all $i\in [k]$. We have already seen the most relevant examples: winning families arising from football are clearly increasing. More generally, such increasing families, commonly referred to as compressed or shifted families, are ubiquitous in extremal combinatorics; see~\citep{ErdosKoRado1961,Kruskal1963,Katona1968,Frankl1987}.

Our main result, stated below, is the combinatorial heart of Theorem~\ref{thm:main-football}. It is reminiscent of the basic fact, see~\citep{HefetzKrivelevichStojakovicSzabo2014} for example, that the first player to move cannot lose the strong (hypergraph) game. There is, however, a small subtlety here: the stronger player is parity-dependent; for picker--chooser games on increasing families, Alice is stronger on odd boards, and Bob is stronger on even boards.

\begin{theorem}\label{thm:parity}
	For each pair of integers $n \ge k \ge 0$ and every increasing $\cF\subset \binom{[n]}{k}$,
	\begin{enumerate}[(i)]
		\item\label{thm:parity-even} if $n$ is even and Alice wins her $\cF$-game, then Bob wins his $\cF$-game, and
		\item\label{thm:parity-odd} if $n$ is odd and Bob wins his $\cF$-game, then Alice wins her $\cF$-game.
	\end{enumerate}
\end{theorem}

It is easy to see that Theorem~\ref{thm:parity} implies Theorem~\ref{thm:main-football}. We also note that monotonicity is essential in Theorem~\ref{thm:parity}; both assertions in the result fail to hold in the absence of such a hypothesis on the winning family.

Theorem~\ref{thm:parity} is still not strong enough to prove itself. However, once we correctly \emph{quantify} the advantage that Alice and Bob possess in the parities where they are respectively stronger, the details become straightforward, as we shall see.

The rest of this paper is organised as follows. Section~\ref{s:margins} introduces the notion of slack that allows us to strengthen and, in turn, prove Theorem~\ref{thm:parity}. The proof of Theorem~\ref{thm:parity} via this stronger result follows in Section~\ref{s:proof}. We conclude in Section~\ref{s:open} with a brief discussion of some open problems and directions for future research.

\section{Winning margins}\label{s:margins}
In this section, we introduce some terminology and the notion of slack that will allow us to sufficiently strengthen Theorem~\ref{thm:parity}.

For integers $n \ge k \ge 0$, we write $[n]$ for the set $\{1,\dots,n\}$ and $\binom{[n]}{k}$ for the family of all $k$-element subsets of $[n]$. For a family $\cF\subset \binom{[n]}{k}$, we write $\Awin(\cF)$ for the truth value of `Alice wins her $\cF$-game', and $\Bwin(\cF)$ for the truth value of `Bob wins his $\cF$-game'.

For the \emph{terminal} games where $k=0$ or $k=n$, we trivially have $\Awin(\cF)=\Bwin(\cF)$, with this common value being determined by whether the unique possible final set lies in $\cF$. Theorem~\ref{thm:parity} is clearly trivial in these terminal cases, so we assume in what follows that $1\le k\le n-1$.

\subsection*{Sections}
For $x\in [n]$, the $x$-sections of $\cF \subset \binom{[n]}{k}$ are
\[
	\cF_x^+ = \{S\setminus \{x\}: x\in S\in \cF\}
\]
and
\[
	\cF_x^- = \{S\in \cF: x\notin S\}.
\]
Since we will need to compare $x$-sections across different $x$, we \emph{standardise} these sections by viewing them as uniform families on the ground set $[n-1]$ under the unique order-preserving bijection from $[n]\setminus \{x\}$ to $[n-1]$; clearly, all the sections of an increasing family are increasing (after standardisation).

In the $\cF$-game, $\cF_x^+$ is the winning family after the first offer $x$ goes to the protagonist, while $\cF_x^-$ is the winning family after $x$ goes to the antagonist. When we write $\Awin(\cF_x^\pm)$ or $\Bwin(\cF_x^\pm)$, the associated games are understood to be played on the standardised sections (on the board $[n-1]$) under the usual conventions: Alice picks first, and Bob chooses first.

Standardising after sectioning off both $x$ and $y$ for any distinct $x,y\in [n]$ gives the same result regardless of the order in which we section off $x$ and $y$. In particular, we have
\beq{eq:commuting-sections}
(\cF_x^\sigma)_y^\tau=(\cF_y^\tau)_x^\sigma
\enq
for any distinct $x,y\in [n]$ and $\sigma,\tau\in \{+,-\}$, whenever both sides are defined.

From the definition of the games, to win her $\cF$-game, Alice needs one good first offer, i.e., an $x \in [n]$ for which Bob wins both his $\cF_x^+$-game and his $\cF_x^-$-game, while to win his $\cF$-game, Bob needs a good response to every first offer, i.e., Alice must win either her $\cF_x^+$-game or her $\cF_x^-$-game for every $x\in [n]$. In other words, we see that
\beq{eq:recursion-A}
\Awin(\cF)
\quad\Longleftrightarrow\quad
\exists x\in [n]:\ \Bwin(\cF_x^+)\wedge \Bwin(\cF_x^-),
\enq
and
\beq{eq:recursion-B}
\Bwin(\cF)
\quad\Longleftrightarrow\quad
\forall x\in [n]:\ \Awin(\cF_x^+)\vee \Awin(\cF_x^-).
\enq

\subsection*{Monotonicity}
Next, we record some straightforward monotonicity properties for picker--chooser games on increasing families.

First, if $\cF,\cG\subset \binom{[n]}{k}$ are increasing families and $\cF\subset \cG$, then
\beq{eq:section-monotonicitybasic}
\Awin(\cF) \Longrightarrow \Awin(\cG)
\qquad \text{and}\qquad
\Bwin(\cF) \Longrightarrow \Bwin(\cG).
\enq

Next, if $\cF$ is increasing and $x<y$ are two distinct elements of $[n]$, then after standardising the $x$-sections and $y$-sections, we have
\beq{eq:section-monotonicity}
\cF_x^+\subset \cF_y^+
\qquad \text{and}\qquad
\cF_x^-\supset \cF_y^-.
\enq

For an increasing family $\cF$, it follows from these basic monotonicity properties that if we write
\begin{align*}
	U_A(\cF) & =\{x\in [n]:\Bwin(\cF_x^+)\},             \\
	L_A(\cF) & =\{x\in [n]:\Bwin(\cF_x^-)\},             \\
	U_B(\cF) & =\{x\in [n]:\Awin(\cF_x^+)\}, \text{ and} \\
	L_B(\cF) & =\{x\in [n]:\Awin(\cF_x^-)\},
\end{align*}
then both $U_A(\cF)$ and $U_B(\cF)$ are upper intervals of $[n]$, and both $L_A(\cF)$ and $L_B(\cF)$ are lower intervals of $[n]$.

\subsection*{Margins}
For an increasing family $\cF$, by \emph{Alice's margin} in her $\cF$-game, we mean the set $U_A(\cF)\cap L_A(\cF)$, and by \emph{Bob's margin} in his $\cF$-game, we mean the set $U_B(\cF)\cap L_B(\cF)$. Alice's margin is precisely the set of her winning first offers in her $\cF$-game. Bob's margin, on the other hand, is the set of first offers to which Bob may respond with \emph{indifference} and still win his $\cF$-game.

In this language, $\Awin(\cF)$ holds if and only if Alice has nonempty margin, i.e.,
\[
	\Awin(\cF) \quad\Longleftrightarrow\quad U_A(\cF)\cap L_A(\cF)\neq \emptyset.
\]
Also, Bob having nonempty margin is sufficient, but not necessary, for $\Bwin(\cF)$. Indeed,
\[
	\Bwin(\cF) \quad\Longleftrightarrow\quad U_B(\cF)\cup L_B(\cF)=[n];
\]
since $U_B(\cF)$ is an upper interval and $L_B(\cF)$ is a lower interval, their union is $[n]$ whenever their overlap, Bob's margin, is nonempty. However, Bob does not need nonempty margin to win: there are increasing families $\cF$ for which $\Bwin(\cF)$ holds even though $U_B(\cF)\cap L_B(\cF)=\emptyset$.

Before we state the strengthening of Theorem~\ref{thm:parity} that we shall prove by induction on the size of the board, we explain why Theorem~\ref{thm:parity} is not strong enough to prove itself.

The even-board implication for $n$ follows directly from the odd-board implication for $n-1$. To see this, suppose that $n$ is even and $\Awin(\cF)$ holds. Then, by~\eqref{eq:recursion-A}, there is an element $a\in [n]$ such that
\[
	\Bwin(\cF_a^+) \qquad\text{and}\qquad \Bwin(\cF_a^-)
\]
both hold. Since $n-1$ is odd, the inductively assumed odd-board implication tells us that
\[
	\Awin(\cF_a^+) \qquad\text{and}\qquad \Awin(\cF_a^-)
\]
both hold. It follows that Bob can win his $\cF$-game by accepting any offer $x \ge a$ and rejecting any offer $x \le a$. Indeed, by monotonicity, for every $x\le a$, we know that $\cF_x^- \supset \cF_a^-$, so $\Awin(\cF_x^-)$ holds, and for every $x\ge a$, we know that $\cF_x^+ \supset \cF_a^+$, so $\Awin(\cF_x^+)$ holds. Hence, for every $x\in [n]$, at least one of $\Awin(\cF_x^+)$ and $\Awin(\cF_x^-)$ holds. By~\eqref{eq:recursion-B}, it follows that $\Bwin(\cF)$ holds.

In contrast, the odd-board implication for $n$ does not follow from the even-board implication for $n-1$ in the same way. If $\Bwin(\cF)$ holds and Bob has nonempty margin, then induction shows that Alice's margin contains Bob's margin, and is therefore nonempty. But if Bob `barely' wins his $\cF$-game, i.e., if $\Bwin(\cF)$ holds but Bob's margin is empty, then Theorem~\ref{thm:parity} gives no clear candidate element in Alice's margin, i.e., no clear candidate for Alice's winning first offer in her own game.

Hinting at what follows, note that the even-from-odd argument above shows a little more than what is needed: if $\Awin(\cF)$ holds on an even board, then not only does $\Bwin(\cF)$ hold, but Bob has nonempty margin.

Some amount of experimentation leads us to a strengthening of Theorem~\ref{thm:parity} that does, in fact, prove itself; the following result asserts that if the weaker player has a winning strategy, then the margin of the stronger player contains the middle of the board.

\begin{theorem}\label{thm:central}
	For each pair of integers $n > k > 0$ and every increasing $\cF\subset \binom{[n]}{k}$, the following hold.
	\begin{enumerate}[(i)]
		\item Suppose $n=2m+1$.  If $\Bwin(\cF)$ holds, then
		      \[
			      \Bwin(\cF_{m+1}^+) \qquad\text{and}\qquad
			      \Bwin(\cF_{m+1}^-)
		      \]
		      both hold.  Equivalently, Alice wins her $\cF$-game by opening with the middle element.

		\item Suppose $n=2m$.  If $\Awin(\cF)$ holds, then
		      \[
			      \Awin(\cF_m^+) \qquad\text{and}\qquad \Awin(\cF_m^-)
		      \]
		      both hold, and
		      \[
			      \Awin(\cF_{m+1}^+) \qquad\text{and}\qquad \Awin(\cF_{m+1}^-)
		      \]
		      both hold.  Equivalently, in Bob's $\cF$-game, Bob can respond indifferently to either middle element and still win.
	\end{enumerate}
\end{theorem}

\section{Proof of the main result}\label{s:proof}

In this section, we prove Theorem~\ref{thm:central} by induction on the size of the board; unlike in the earlier discussion, we now need to consider sections of sections.

\begin{proof}[Proof of Theorem~\ref{thm:central}] We prove the result by induction on $n$.  The case $n=1$ is vacuous, so we may assume that $n\ge 2$ and that the result holds for all smaller values of $n$.

	\noindent\textbf{Boundary layers.}
	Before the main induction step, we remove the two boundary layers where one of the first sections is terminal.

	\begin{claim}\label{cl:boundary}
		Theorem~\ref{thm:central} holds when $k=1$ or $k=n-1$.
	\end{claim}

	\begin{proof}
		First suppose $k=1$.  An increasing family of singletons has the form
		\[
			\cT_n(t)=\bigl\{\{r\}: t\le r\le n\bigr\}
		\]
		for some $1\le t\le n+1$, where $t=n+1$ gives the empty family.  It is straightforward to verify that
		\beq{eq:k-one-formulas}
		\Awin(\cT_n(t)) \Longleftrightarrow t\le \left\lceil \frac n2\right\rceil,
		\qquad
		\Bwin(\cT_n(t)) \Longleftrightarrow t\le \left\lfloor \frac n2\right\rfloor+1.
		\enq

		If $n=2m+1$ and $\Bwin(\cT_n(t))$ holds, then $t\le m+1$.  The section $(\cT_n(t))_{m+1}^+$ is terminal and true, while $(\cT_n(t))_{m+1}^-$ is $\cT_{2m}(t)$ after standardisation; by~\eqref{eq:k-one-formulas}, $\Bwin(\cT_{2m}(t))$ holds.  This proves the odd central-move assertion for $k=1$.

		If $n=2m$ and $\Awin(\cT_n(t))$ holds, then $t\le m$.  For the two middle elements $m$ and $m+1$, both sections $(\cT_n(t))_m^+$ and $(\cT_n(t))_{m+1}^+$ are terminal and true, while $(\cT_n(t))_m^-$ and $(\cT_n(t))_{m+1}^-$ are both $\cT_{2m-1}(t)$ after standardisation; again~\eqref{eq:k-one-formulas} gives $\Awin(\cT_{2m-1}(t))$.  This proves the even central-move assertion for $k=1$.

		The case $k=n-1$ is equivalent by complementation.  Indeed, for an increasing family $\cF\subset \binom{[n]}{n-1}$, set
		\[
			\cF^*=\bigl\{\{r\}: [n]\setminus\{r\}\notin \cF\bigr\}.
		\]
		Then $\cF^*$ is an increasing singleton family, and since complementation swaps protagonist and antagonist, we see that
		\[
			\Awin(\cF)\Longleftrightarrow \neg\Bwin(\cF^*)
			\qquad\text{and}\qquad
			\Bwin(\cF)\Longleftrightarrow \neg\Awin(\cF^*).
		\]
		Moreover, after deleting $x$, complementation interchanges the two sections:
		\[
			(\cF_x^+)^*=(\cF^*)_x^-
			\qquad\text{and}\qquad
			(\cF_x^-)^*=(\cF^*)_x^+,
		\]
		with the evident terminal interpretation. Thus, the case where $k=n-1$ is equivalent to, and follows from, the case where $k=1$.
	\end{proof}

	We now carry out the induction step.  The boundary layers $k=1$ and $k=n-1$ are covered by Claim~\ref{cl:boundary}, so we may now assume that
	\beq{eq:nonterminal-range}
	2\le k\le n-2.
	\enq
	This is exactly the range in which every section to which we apply the induction hypothesis is again a genuine nonterminal game; in particular, at this point $n\ge 4$.

	\noindent\textbf{Odd boards from even boards.} Let $n=2m+1$ and $2\le k\le 2m-1$. We need to show, for any increasing $\cF \subset \binom{[n]}{k}$, that if $\Bwin(\cF)$ holds, then Alice can win her $\cF$-game by opening with the middle element $m+1$, i.e., that
	\[
		\Bwin(\cF_{m+1}^+) \qquad\text{and}\qquad \Bwin(\cF_{m+1}^-)
	\]
	both hold. To show this, we need to show that Bob has a good response for every first offer $y\in [n]\setminus \{m+1\}$ in both his $\cF_{m+1}^+$-game and his $\cF_{m+1}^-$-game (on $[2m]$ after standardisation), i.e., that for every $y\in [n]\setminus \{m+1\}$, we have
	\beq{eq:odd-goal-1}
	\Awin((\cF_{m+1}^+)_{y}^+) \qquad\text{or}\qquad \Awin((\cF_{m+1}^+)_{y}^-),
	\enq
	as well as
	\beq{eq:odd-goal-2}
	\Awin((\cF_{m+1}^-)_{y}^+) \qquad\text{or}\qquad \Awin((\cF_{m+1}^-)_{y}^-).
	\enq

	Now assume that $\Bwin(\cF)$ holds.  This means exactly that, for every first offer $y\in [n]$, Bob has at least one good response: at least one of
	\beq{eq:odd-main-recursion}
	\Awin(\cF_y^+) \qquad\text{or}\qquad \Awin(\cF_y^-)
	\enq
	holds.

	Fix $y\in [n]\setminus \{m+1\}$.  We will show that both~\eqref{eq:odd-goal-1} and~\eqref{eq:odd-goal-2} hold for this $y$.  After deleting $y$, the surviving element corresponding to $m+1$ is one of the two middle elements of the board: on the board $[2m]$, it has rank $m$ if $y<m+1$, and rank $m+1$ if $y>m+1$.

	First, suppose that $\Awin(\cF_y^+)$ holds on the residual board.  The induction hypothesis for this smaller even board says, in particular, that both sections at the surviving copy of $m+1$ are Alice-true:
	\[
		\Awin\bigl((\cF_y^+)_{m+1}^+\bigr)
		\qquad\text{and}\qquad
		\Awin\bigl((\cF_y^+)_{m+1}^-\bigr)
	\]
	hold.  Equivalently, by the commutation of sections, both
	\[
		\Awin\bigl((\cF_{m+1}^+)_y^+\bigr)
		\qquad\text{and}\qquad
		\Awin\bigl((\cF_{m+1}^-)_y^+\bigr)
	\]
	hold, demonstrating that the first alternatives in~\eqref{eq:odd-goal-1} and~\eqref{eq:odd-goal-2} both hold for this $y$.

	Next, suppose instead that $\Awin(\cF_y^-)$ holds.  The argument is the same, with the sign of the $y$-section reversed; it gives
	\[
		\Awin\bigl((\cF_{m+1}^+)_y^-\bigr)
		\qquad\text{and}\qquad
		\Awin\bigl((\cF_{m+1}^-)_y^-\bigr)
	\]
	so the second alternatives in~\eqref{eq:odd-goal-1} and~\eqref{eq:odd-goal-2} hold for this $y$.

	By~\eqref{eq:odd-main-recursion}, one of these two cases applies to every $y\in [n]\setminus \{m+1\}$, so both~\eqref{eq:odd-goal-1} and~\eqref{eq:odd-goal-2} hold for every such $y$. This completes the odd-from-even argument.

	\noindent\textbf{Even boards from odd boards.}  Let $n=2m$ and $2\le k\le 2m-2$. We need to show, for any increasing $\cF \subset \binom{[n]}{k}$, that if $\Awin(\cF)$ holds, then Bob can win his $\cF$-game while responding with indifference to either middle first offer $m$ or $m+1$, i.e., that
	\[
		\Awin(\cF_{m}^+) \qquad\text{and}\qquad \Awin(\cF_{m}^-),
	\]
	as well as
	\[
		\Awin(\cF_{m+1}^+) \qquad\text{and}\qquad \Awin(\cF_{m+1}^-),
	\]
	all hold.

	Now assume that $\Awin(\cF)$ holds.  This means exactly that Alice has a good first offer $a\in [n]$ for which she wins her $\cF$-game regardless of how Bob responds, i.e., for which both
	\beq{eq:good-a}
	\Bwin(\cF_a^+) \qquad\text{and}\qquad \Bwin(\cF_a^-)
	\enq
	hold.  To show that the four displayed goals above hold, we will have to distinguish two cases, depending on the position of $a$ relative to the middle of the board.

	First, suppose that $a\le m$.  After deleting $a$, the element corresponding to $m+1$ is the middle element of the residual odd board. Bob wins both his $\cF_a^+$-game and his $\cF_a^-$-game on the residual odd board, so the induction hypothesis applied to these games tells us that Alice can win both her $\cF_a^+$-game and her $\cF_a^-$-game by opening with the middle element $m+1$, i.e., that
	\beq{eq:case-one-middle-sections}
	\begin{aligned}
		 & \Bwin\bigl((\cF_a^+)_{m+1}^+\bigr)
		\qquad\text{and}\qquad
		\Bwin\bigl((\cF_a^+)_{m+1}^-\bigr), \text{ and } \\
		 & \Bwin\bigl((\cF_a^-)_{m+1}^+\bigr)
		\qquad\text{and}\qquad
		\Bwin\bigl((\cF_a^-)_{m+1}^-\bigr).
	\end{aligned}
	\enq
	all hold.

	In particular, by the recursion for $\Awin$, both
	\beq{eq:case-one-a-sections}
	\Awin(\cF_a^+) \qquad\text{and}\qquad \Awin(\cF_a^-)
	\enq
	hold.  Since the sections $\cF_i^+$ increase with $i$, and since $a\le m<m+1$, monotonicity shows that
	\beq{eq:case-one-plus}
	\Awin(\cF_m^+) \qquad\text{and}\qquad \Awin(\cF_{m+1}^+)
	\enq
	are both true.

	It remains to show that Alice wins her $\cF_m^-$-game and her $\cF_{m+1}^-$-game as well. In Alice's $\cF_{m+1}^-$-game, let Alice open with $a$.  If she receives $a$, her residual winning family is
	\[
		(\cF_{m+1}^-)_a^+ = (\cF_a^+)_{m+1}^-,
	\]
	and if she does not receive $a$, her residual winning family is
	\[
		(\cF_{m+1}^-)_a^- = (\cF_a^-)_{m+1}^-.
	\]
	It follows from~\eqref{eq:case-one-middle-sections} that
	\beq{eq:case-one-m-plus-one-minus}
	\Awin(\cF_{m+1}^-).
	\enq
	Finally, we know that $\cF_m^-\supset \cF_{m+1}^-$ since $\cF$ is increasing, and~\eqref{eq:case-one-m-plus-one-minus} gives us
	\beq{eq:case-one-m-minus}
	\Awin(\cF_m^-)
	\enq
	by monotonicity. The assertions~\eqref{eq:case-one-plus},~\eqref{eq:case-one-m-plus-one-minus}, and~\eqref{eq:case-one-m-minus} prove all four goals when $a\le m$.

	Next, suppose instead that $a\ge m+1$.  This is the mirror image of the previous paragraph.  After deleting $a$, the element corresponding to $m$ is the middle element of the residual odd board, so the same argument with $m$ and $m+1$ interchanged, and with $+$ and $-$ reversed, gives
	\[
		\Awin(\cF_m^-) \qquad\text{and}\qquad \Awin(\cF_{m+1}^-),
	\]
	and then
	\[
		\Awin(\cF_m^+) \qquad\text{and}\qquad \Awin(\cF_{m+1}^+).
	\]
	Thus, all four goals hold when $a\ge m+1$ as well, completing the even-from-odd argument.

	The boundary-layer claims, the odd-from-even argument, and the even-from-odd argument taken together complete the proof by induction.
\end{proof}

Theorem~\ref{thm:parity} is of course immediate from Theorem~\ref{thm:central}, and applying Theorem~\ref{thm:parity} to the increasing winning family associated with the ordered board in football gives Theorem~\ref{thm:main-football}.

\section{Conclusion}\label{s:open}
While picker--chooser games are primarily interesting from the perspective of positional game theory, the mechanics studied here are also strongly reminiscent of ``I-cut-you-choose'' procedures in fair division; see~\citep{Steinhaus1948} for a classical example. We therefore wonder what one can say about alternating picker--chooser games with winning families more general than the monotone ones addressed by Theorem~\ref{thm:parity}.

As mentioned earlier, there appears to be no general analogue of Theorem~\ref{thm:parity} for arbitrary winning families. Concretely, it is not hard to verify (with the aid of a computer) the following: in $\binom{[7]}{3}$, for
\[
	\begin{aligned}
		\cG=\{ & 123,124,127,136,137,146,147,157,167, \\
		       & 234,245,246,247,257,267,345,457\},
	\end{aligned}
\]
we have $\Bwin(\cG)$ but not $\Awin(\cG)$, while in $\binom{[8]}{4}$, for
\[
	\begin{aligned}
		\mathcal H=\{ & 1235,1237,1238,1257,1258,1267,1345,1346,1348, \\
		              & 1356,1357,1358,1367,1378,1457,1458,1467,1478, \\
		              & 1567,1568,1578,2345,2357,2378,2457,2568,3456, \\
		              & 3457,3478,3567,3568,3578,3678,4578,5678\}
	\end{aligned}
\]
we have $\Awin(\mathcal H)$ but not $\Bwin(\mathcal H)$.

We wonder, however, whether one can substitute symmetry for monotonicity. For example, if the board consists of the edges of the complete graph $K_n$ and $\cF$ is any $S_n$-invariant property of $n$-vertex graphs, does the alternating picker--chooser $\cF$-game admit a parity-rule, or some other robust outcome principle?

\section*{Acknowledgements}
The author thanks Tom Eccles, Imre Leader, Jason Long, Joel Spencer and Peter Winkler for several enjoyable conversations about picker--chooser games. The author also acknowledges support from NSF grant DMS-2237138 and a Sloan Research Fellowship.

\bibliographystyle{amsplain}
\bibliography{football}

\end{document}